\documentclass{amsart}
\usepackage{amsfonts}
\usepackage{latexsym}
\usepackage{hyperref}
\usepackage{amssymb}
\usepackage{amsmath}
\usepackage{enumerate}
\usepackage{latexsym,color,amsmath,amsthm,amssymb,amscd,amsfonts}
\usepackage{graphicx}
\usepackage[normalem]{ulem}

\usepackage[utf8]{inputenc}
\usepackage[T1]{fontenc}

\newcommand{\R}{\mathbb R}

\newtheorem{thm}{Theorem}[section]

\newtheorem{lemma}[thm]{Lemma}

\newtheorem{propo}[thm]{Proposition}
\theoremstyle{remark}
\newtheorem*{rmk}{Remark}

\begin{document}


\title[Perelman-Pukhov quotient: better and optimal bounds]{On the Perelman-Pukhov quotient of successive radii: better and asymptotically optimal bounds}

\author[B. Gonz\'alez Merino]{Bernardo Gonz\'alez Merino}
\email{bgmerino@um.es}
\address{Department of Engineering and Technology of Computers, area of Applied Mathematics.
Universidad de Murcia, 
Spain}

\author[B. Mar\'in Gimeno]{Beatriz Mar\'in Gimeno}
\email{b.maringimeno@um.es}
\address{Department of Mathematics. University of Murcia, Spain}

\author[M. Runge]{Mia Runge}
\email{mia.runge@tum.de}
\address{Technical University of Munich, School of Computation, Information, and Technology, Department of Mathematics}


\thanks{The first author is partially supported by Ministerio de Ciencia, Innovación y Universidades project PID2022-136320NB-I00/AEI/10.13039/501100011033/FEDER, UE. The second author is partially supported by PID2021-124157NB-I00 funded by MCIN/AEI
/10.13039/501100011033/ ‘ERDF A way of making Europe’, Spain, and by Contratos Predoctorales FPU-Universidad de Murcia 2024, Spain.}


\begin{abstract}
Perel'man in $1987$ and independently Pukhov in $1979$ proved that the quotient between the $(n-i+1)$-th successive outer radius and the $i$-th successive inner radius of a convex body in $n$-dimensions is not larger than $i+1$. Apart from the solved cases by
Jung $1901$ ($i=1$) and Steinhagen $1921$ ($i=n$),
only Perel'man ($1987$, $n=3$, $i=2$) and Gonz\'alez Merino ($2017$, $n\geq 4$, $i=2$ and $i=n-1$) provided small improvements that beat this bound. 


In this paper, we obtain sharper inequalities using relations between these inner and outer measures with the diameter and minimal width. 
We improve the current bounds in the following cases: $i=3$ when $4\leq n \leq 8$, $i=4$ when $n=5$, $6$, $i=5$ when $n=6$, $i=6$ when $n=7$, and for every $i\geq n-\Theta(\log n)$. 
Notably, our bounds 
provide the right order in $n$ when $i=n-m$, with $m$ constant and $n$ arbitrarily large. 
Additionally, we improve the case $n=5$, $i=3$ even further by refining an idea of Perel'man and using the optimal lower bound of the inradius in terms of the circumradius and the diameter in $3$-space (see \cite{BGR25}).
\end{abstract}

\maketitle

\section{Introduction}


Let $K\subset\mathbb R^n$ be a convex body, i.e.~a convex and compact set in $\mathbb R^n$ with non-empty interior. For every $1\leq i\leq n$, let $R_i(K)$ be the smallest radius of a solid cylinder of $i$-dimensional spherical cross-section containing $K$, and let $r_i(K)$ be the largest radius of an $i$-dimensional Euclidean disc contained in $K$. Perel'man \cite{Pe87} (see also Pukhov \cite{Pu79}) studied the relation between those quantities and proved that
\begin{equation}\label{eq:PerelPuk}
    \frac{R_{n-i+1}(K)}{r_i(K)} \leq i+1
\end{equation}
for every $i\in[n]:=\{1,\dots,n\}$. However, the above estimate is far from being sharp. Indeed, the cases $i=1$ and $i=n$ are classic results from convex geometry. The former was studied by Jung \cite{J01}, who showed that
\begin{equation}\label{eq:Jung}
    \frac{R_n(K)}{r_1(K)} \leq \sqrt{\frac{2n}{n+1}}.
\end{equation}
The latter was formalized by Steinhagen \cite{St21}, proving that
\begin{equation}\label{eq:Steinhagen}
    \frac{R_1(K)}{r_n(K)} \leq \left\{
    \begin{array}{cc}
        \sqrt{n} &  \text{if }n\text{ is odd,} \\
        \frac{n+1}{\sqrt{n+2}} & \text{ otherwise.} 
    \end{array}\right.
\end{equation}
Let us also remember that equality holds both in \eqref{eq:Jung} and \eqref{eq:Steinhagen} when $K$ is an $n$-dimensional regular simplex $S_n$. This is partly the reason why it is conjectured that the simplex is an extreme body for the optimal upper bound of the ratio within \eqref{eq:PerelPuk}. If $i=1$ or $i=n$, then $S_n$ attains equality in \eqref{eq:Jung} and \eqref{eq:Steinhagen}, respectively.  If $i=2$ and $n$ is even, then
\[
\frac{R_{n-1}(S_n)}{r_2(S_n)} = \frac{(2n-1)\sqrt{3}}{\sqrt{2n(n+1)}},
\]
and in the remaining cases
\begin{equation}\label{eq:simplex}
\frac{R_{n-i+1}(S_n)}{r_i(S_n)} = \sqrt{1-\frac{i}{n+1}}\sqrt{i(i+1)}
\end{equation}
(see \cite{Br05}). There is a limited amount of results that have improved the estimates in \eqref{eq:PerelPuk}. On the one hand, in \cite{GM17} it was shown (see Theorems 1.2 and 1.3) that
\[
\frac{R_{n-1}(K)}{r_{2}(K)} \leq 2\sqrt{2}\sqrt{\frac{n-1}{n}} \quad\text{and}\quad 
\frac{R_{2}(K)}{r_{n-1}(K)} \leq 2\sqrt{2}\sqrt{n}
\]
for every $n\geq 3$. 

In this paper, we generalize these ideas to improve \eqref{eq:PerelPuk} in two ways.
\begin{thm}\label{thm:Main}
    Let $K\subset\mathbb R^n$ be a convex body. For every $i\in\{2,\dots,n-1\}$ we have that
    \[
    \frac{R_{n-i+1}(K)}{r_i(K)} \leq \min\left\{ i \sqrt{2}\sqrt{1-\frac{1}{n-i+2}}\,, 2^{n-i}\sqrt{(n-i+1)(i+1)} \right\}.
    \]
\end{thm}
The behavior of the estimates in Theorem \ref{thm:Main} differ a bit. The first estimate within the minimum provides better estimates for a finite amount of cases $(n,i)$: when $i=3$ and $n=4,\dots,9$, when $i=4$ and $n=5,6$, when $i=5$ and $n=6$, and when $i=6$ and $n=7$. The second estimate provides a better estimate than \eqref{eq:PerelPuk} whenever $i \geq n-\Theta(\log n)$. Moreover, for every $i=n-m$ for a fixed $m\in\mathbb N$, Theorem \ref{thm:Main} provides the correct order $O(\sqrt{n})$ of the upper bound for the quotient of the radii, beating the linear order $O(n)$ provided by \eqref{eq:PerelPuk}.

In his former paper, Perel'man \cite{Pe87} developed an idea to further improve the estimate in the particular case of $n=3$ and $i=2$:
\begin{equation}\label{eq:Pereln3i2}
    \frac{R_2(K)}{r_2(K)} <2.151.
\end{equation}
The argument is built on a somewhat hidden and intriguing idea: to use a lower bound for the inradius of a planar convex body in terms of its diameter and circumradius, proven by Santal\'o \cite{Sa61} in 1961. 

Very recently, a sharp estimate in dimension $3$, analogous to that of Santal\'o, has been shown (see \cite[Thm.~1.1 (2)]{BGR25}). This enables us to further improve the case of $n=5$ and $i=3$.
\begin{thm}\label{thm:53}
    Let $K\subset\mathbb R^5$ be a convex body. Then
    \[
    \frac{R_3(K)}{r_3(K)} < 3.518.
    \]
\end{thm}
When $K$ is centrally symmetric, further estimates for the corresponding ratio are provided in \cite{GM17,Pu79}. 

For more information on successive radii, their size for particular bodies, as well as computational aspects of these radii, we refer to \cite{Ba92, BH92, BH93, Br05, BrK11,BrT06, GK92, GK93, H92, Vr24}. Their relation with other measures has been studied in \cite{BH92, HH08, HH09}, their behavior with respect to other binary operations in \cite{CYL16, GH14,GH12}, and their extensions to containers different from the Euclidean ball in \cite{GK92, Ja17}. Moreover, quotients of different radii have been studied in \cite{BH93, BrK11, GH13, GK92, H92}. We would like to point out that successive radii are particular cases of the so-called Gelfand and Kolmogorov numbers in Banach Space Theory (cf. \cite{CHP13, GHH15, Pi85}), and are widely used in Approximation Theory. Recently, analogous definitions for mean radii have been developed (cf.~\cite{AHS18,ADHP14,LS24}).

The paper is organized as follows. In Section \ref{sec:notations} we introduce further notations and other preliminaries required for the development of the results. In Section \ref{sec:smallValues} we make use of the diameter to derive the first estimate in Theorem \ref{thm:Main} as well as Theorem \ref{thm:53} by using the Santal\'o type inequality derived in \cite{BGR25}. Finally, in Section \ref{sec:largeValues}, we use the minimal width to derive the second estimate in Theorem \ref{thm:Main}.

\section{Further notation and preliminaries}\label{sec:notations}

Let $K\subset\mathbb R^n$ be a convex body, i.e.~a convex and compact set in $\mathbb R^n$ with non-empty interior and denote by $\mathcal K^n$ the set of all convex bodies in $\mathbb R^n$. Let $\|x\|_2:=\sqrt{x_1^2+\cdots+x_n^2}$ be the Euclidean norm of $x=(x_1,\dots,x_n)$ and $B_n=\{x\in\mathbb R^n:\|x\|_2\leq 1\}$ be the Euclidean unit ball. For any $X\subset\mathbb R^n$ let $\mathrm{int}(X)$ and $\mathrm{bd}(X)$ be the interior and the boundary of $X$. Let $e^1,\dots,e^n$ be the canonical basis of $\mathbb R^n$.

Let $\mathcal A^n_i$ be the set of $i$-dimensional affine subspaces in $\mathbb R^n$. For every $X\subset \mathbb R^n$, let $\langle X\rangle$, $\mathrm{aff}(X)$, and $\mathrm{conv}(X)$ be the linear, affine, and convex hull of $X$, respectively. If $X\subset\mathbb R^n$, we define $\mathrm{dim}(X):=\mathrm{dim}(\mathrm{aff}(X))$. Moreover, let $X^\bot=\{y\in\mathbb R^n:y^{\top}x=0,\,\forall\,x\in X\}$ be the orthogonal complement of $X$. For any $L\in\mathcal A^n_i$, the orthogonal projection of $K$ onto $L$ is denoted by $P_L(K)$. We say that $\mathrm{relbd}(K)$ is the relative boundary of $K$, i.e.~the boundary of $K$ measured within $\mathrm{aff}(K)$,

For every $X,Y\subset\mathbb R^n$, let $X+Y:=\{x+y:x\in X,y\in Y\}$ be the Minkowski sum of $X$ and $Y$. When $X=\{x\}$, for some $x\in\mathbb R^n$, we write $X+Y=x+Y$. For any $t\in\mathbb R^n$, let $tX:=\{tx:x\in X\}$ be a rescaling of $X$ by $t$. For any $x,y\in\mathbb R^n$, let $[x,y]:=\{(1-\lambda) x+\lambda y:\lambda\in[0,1]\}$ be the line segment of endpoints $x$ and $y$. 

We say that $Q$ is a parallelotope if there exist $u^1,\dots,u^m \in\mathbb R^n$ linearly independent vectors such that $Q=x+\sum_{i=1}^m[0,u^i]$, for some $x\in\mathbb R^n$.

The circumradius of $K$ is given by $R(K)=\min\{\rho>0:x+K\subseteq \rho B_n,\,x\in\mathbb R^n\}$, the inradius by $r(K)=\max\{\rho>0:\rho B_n\subseteq x+K,\,x\in\mathbb R^n\}$, and the diameter by $D(K)=\max\{\|x-y\|_2:x,y\in K\}$. 

The support function of $K$ is defined by $h_K(u):=\sup\{u^{\top}x:x\in K\}$, for every $u\in\mathbb R^n\setminus\{0\}$. If $H\in\mathcal A^n_{n-1}$, we say that $H$ supports $K$ if $K$ is contained in one of the halfspaces determined by $H$ and $K\cap H\neq \emptyset$. In particular, $\{x\in\mathbb R^n:x^{\top}y=h_K(y)\}$ is a supporting hyperplane of $K$ for every $y\in\mathbb R^n\setminus\{0\}$.

The width of $K$ in the direction $u\in\mathbb R^n$, $\|u\|_2=1$, is given by $h_K(u)+h_K(-u)$. The minimal width of $K$ is defined by
\[
w(K)=\min\{h_K(u)+h_K(-u):\|u\|_2=1,\,u\in\mathbb R^n\}.
\]
A standard property of the minimal width is that for every $K\in\mathcal K^n$ there exist $x,y\in\mathrm{bd}(K)$ and $a\in\mathbb R^n$ with $\Vert a\Vert_2=1$ such that
\begin{equation}\label{eq:widthParallel}
w(K)=h_K(a)+h_K(-a)=a^{\top}x-a^{\top}y
\end{equation}
with $x-y\in\langle a\rangle$ (see \cite{GK92}).

For any $L\in\mathcal A^n_i$ such that $K\subset L$, we denote by $r_L(K)$ to be the inradius of $K$ measured within the subspace $L$. The same applies to $w_L(K)$. Note that $R_L(K)=R(K)$ and $D_L(K)=D(K)$.

The circumradius can be characterized by a finite amount of touching points between the boundaries of a convex body and its circumball.
\begin{propo}
    \label{prop:opt}
    Let $K\in\mathcal K^n$ be such that $K\subset B_n$. Then, the following are equivalent.
    \begin{enumerate}[i)]
        \item $R(K)=1$.
        \item There exist $p^1,\hdots, p^k\in K\cap \mathrm{bd}(B_n)$, for some $k\in\{2,\dots, n+1\}$ such that $0\in\mathrm{conv}(\{p^1,\hdots,p^k\})$.
    \end{enumerate}
\end{propo}

The $i$-th successive outer radius of $K$ is given by $R_i(K)=\min\{R(P_L(K)):L\in\mathcal A^n_i\}$, and the $i$-th successive inner radius of $K$ is given by $r_i(K)=\max\{r_L(K\cap L):L\in\mathcal A^n_i\}$. Notice that $R_n(K)=R(K)$, $r_n(K)=r(K)$, $2r_1(K)=D(K)$ and $2R_1(K)=w(K)$. Moreover, $R_1(K)\leq\cdots\leq R_n(K)$ and $r_n(K)\leq\cdots\leq r_1(K)$, and each successive radius is a monotonically increasing function with respect to set inclusion that is $1$-homogeneous.

We say that $K\in\mathcal K^n$ is a centrally symmetric set if there exists $x\in K$ such that $-x+K=-(-x+K)$. It is well known that in such case $2r(K)=w(K)$ and $2R(K)=D(K)$. A set $K\in\mathcal K^n$ is said to be a polytope if there exist $x^1,\dots,x^m\in\mathbb R^n$ such that $K=\mathrm{conv}(\{x^i:i\in[m]\})$. A facet of a polytope $K\in\mathcal K^n$ is a maximal subset with respect to set inclusion of $\mathrm{bd}(K)$ of dimension $n-1$.

Let $\mathrm{vol}(K)$ be the volume or Lebesgue measure of $K$. More generally, if $K\subset L$ for some $L\in\mathcal A^n_i$, we say that $\mathrm{vol}_i(K)$ is the volume of $K$ measured within the subspace $L$. In this general setting, the $i$-th volume $\mathrm{vol}_i$ is monotonically increasing with respect to set inclusion and homogeneous of degree $i$, i.e.~$\mathrm{vol}_i(tK)=t^i\mathrm{vol}_i(K)$ for every $t\geq 0$, $K\in\mathcal K^n$, $L\in\mathcal A^n_i$ with $K\subset L$.

A 3-dimensional version of Santaló's inequality \cite{Sa61} 
was recently shown in \cite[Thm.~1.1]{BGR25}. For three dimensional convex bodies $K\in\mathcal K^3$ with $D(K)\leq \sqrt{3}R(K)$, we have
\begin{equation}
    \label{eq:rDR_inequality}
       r(K) \geq \frac{D(K)^2\sqrt{3R(K)^2-D(K)^2}}{4R(K)\sqrt{3R(K)^2-D(K)^2}+\sqrt{3}(4R(K)^2-D(K)^2)}. 
\end{equation}

A technical result we need for connecting the inradius of a lower-dimensional $K$ and some of its projections is the following: for every $K\in\mathcal K^n$ with $\mathrm{dim}(K)=i$ and $L\in\mathcal A^n_i$, we have that
\begin{equation}\label{eq:inradiusProyectionidim}
    r_L(P_L(K)) \leq r_{\mathrm{aff}(K)}(K)
\end{equation}
(see \cite[Lemma 4.1]{GM17}).

\section{Estimates for small values of $i$}\label{sec:smallValues}

We start the section by proving a fundamental lemma that connects the diameter of a suitable projection of a convex body $K$ with the $i$-th 
inner radius of $K$. 

\begin{lemma}\label{lem:Diamr_i}
    Let $K\in\mathcal K^n$ and $i\in\{2,\dots,n\}$. Then, there exists $H\in\mathcal A^n_{i-1}$ such that
    \[
    D(P_{H^\bot}(K)) \leq 2i r_i(K).
    \]
\end{lemma}

\begin{proof}

For any set of linearly independent vectors $u^1,\dots,u^m \in\mathbb R^n$, $m\leq n$, let $\bigwedge_{j=1}^{m} u^j := \sum_{j=1}^m[0,u^j]$ be the $m$-dimensional parallelotope spanned by $u^1,\dots,u^m$.

Let $\left\{u^{j_k}\in K:\,k=1,2,\,j=1,\ldots,i-1\right\}$ be such that the $(i-1)$-dimensional volume of $Q:=\bigwedge_{j=1}^{i-1}\widetilde{u}^j$ is maximal among all possible choices $u^{j_k}\in K$, $k=1,2,\,j=1,\ldots,i-1$, where $\widetilde{u}^j:=u^{j_1}-u^{j_2}$, for every $j=1,\dots,i-1$. Let us furthermore consider $H:=\left\langle\left\{\widetilde{u}^j:\,j=1,\ldots,i-1\right\}\right\rangle$.

Let $p^1,p^2\in H^\perp$ be such that $D(P_{H^{\perp}}(K))=\left\Vert p^1-p^2\right\Vert_2$. Moreover, let $q^1,q^2\in K$ be such that $P_{H^\perp}(q^1)=p^1$ and $P_{H^\perp}(q^2)=p^2$. We also denote $\widetilde{u}^i:=q^1-q^2$. Let us define
\[
P:=\frac1i \cdot \mathrm{conv}\left(\left\{q^\ell+\sum_{j=1}^{i-1}u^{j_{k_j}}: \,\ell,k_j\in\{1,2\}\right\}\right)\subseteq K.
\]
We now observe that $P$ is a $i$-dimensional parallelotope. On the one hand, 
\begin{eqnarray*}
        \frac{1}{i}\widetilde{u}^i=\frac{1}{i}\left(q^1+\sum_{j=1}^{i-1}u^{j_{k}}\right)-\frac{1}{i}\left(q^2+\sum_{j=1}^{i-1}u^{j_{k}}\right);
\end{eqnarray*}
on the other hand
\begin{eqnarray*}
        \frac{1}{i}\widetilde{u}^s=\frac{1}{i}\left(q^\ell+u^{s_1}+\sum_{j\in[i-1]\setminus s}u^{j_{k}}\right)-\frac{1}{i}\left(q^\ell+u^{s_2}+\sum_{j\in[i-1]\setminus s}u^{j_{k}}\right)
\end{eqnarray*}
for every $s=1,\ldots,i-1$. Thus, 
    \[
    P=\frac1i\left(q^2+\sum_{j=1}^{i-1}u^{j_2}\right)+\bigwedge_{j=1}^{i} \left(\frac{1}{i} \widetilde{u}^j\right).
    \]
    Since $P$ is a $i$-dimensional parallelotope, and thus centrally symmetric, it follows that 
    \begin{equation}\label{eq:centralSymm}
    r_{\mathrm{aff}(P)}(P)=\frac{w_{\mathrm{aff}(P)}(P)}{2}.
        \end{equation}
We now compute the minimal width of $P$. Let $\sigma\subseteq[i]$ with $|\sigma|=i-1$. Let $h_\sigma$ be the height of $P$ (within $\mathrm{aff}(P)$) orthogonal to the subspace $H_\sigma:=\langle \{\widetilde{u}^j:j\in\sigma\}\rangle$. Evidently, we have that $$w_{\mathrm{aff}(P)}(P)=\min\left\{h_{\sigma}:\,\sigma\subseteq \{1,\ldots,i\}\text{ with } |\sigma|=i-1\right\}.$$

Using the fact that $P$ is a parallelotope and since the facets of $P$ are generated by the vectors $\{\frac1i \widetilde{u}^j:j\in\sigma\}$, for every $\sigma\subseteq[i]$ with $|\sigma|=i-1$, we have that
    \begin{eqnarray*}
            \mathrm{vol}_{i}(P)=h_\sigma\mathrm{vol}_{i-1}\left(\bigwedge_{j\in \sigma}\left(\frac{1}{i}\widetilde{u}^j\right)\right).
    \end{eqnarray*}
    Note that by the homogeneity of the volume, we have that
    \[
    \mathrm{vol}_{i-1}\left(\bigwedge_{j\in \sigma}\left(\frac{1}{i}\widetilde{u}^j\right)\right) = \frac{1}{i^{i-1}}\mathrm{vol}_{i-1}\left(\bigwedge_{j\in \sigma}\widetilde{u}^j\right)
    \]
In particular, we would have that
\begin{eqnarray}\label{eq:maxofh}
        h_\sigma=\frac{\mathrm{vol}_{i-1}\left(\bigwedge_{j=1}^{i-1}\left(\frac{1}{i}\widetilde{u}^j\right)\right)}{\mathrm{vol}_{i-1}\left(\bigwedge_{j\in \sigma}\left(\frac{1}{i}\widetilde{u}^j\right)\right)} h_{[i-1]}
        =\frac{\mathrm{vol}_{i-1}\left(Q\right)}{\mathrm{vol}_{i-1}\left(\bigwedge_{j\in \sigma}\widetilde{u}^j\right)} h_{[i-1]}
        \geq h_{[i-1]},
    \end{eqnarray}
    where the last inequality follows from the maximality of $Q$.

    Note also that 
    \begin{eqnarray}\label{eq:hisdiam}
        h_{[i-1]}=\left\Vert P_{H^\perp}\left(\frac{1}{i}\left(q^1-q^2\right)\right)\right\Vert_2=\frac{1}{i}\left\Vert p^1-p^2\right\Vert_2
        =\frac{1}{i}D(P_{H^\perp}(K)).
    \end{eqnarray}
    We can now conclude that
    \begin{eqnarray*}
        r_i(K) \geq r_{\mathrm{aff}(P)}\left(K\cap \mathrm{aff}(P)\right)\geq r_{\mathrm{aff}(P)}(P)=\frac{w_{\mathrm{aff}(P)}(P)}{2} = \frac{h_{[i-1]}}{2}   
        =   \frac{D(P_{H^\perp}(K))}{2i},
    \end{eqnarray*}
    where above we used the fact that $P\subset K$, the monotonicity of the inradius, \eqref{eq:centralSymm}, \eqref{eq:maxofh}, and \eqref{eq:hisdiam}. This concludes the proof.
\end{proof}

As a corollary, we are now able to show the first inequality from Theorem \ref{thm:Main}.
\begin{proof}[Proof of left-hand-side of Theorem \ref{thm:Main}]
    Let $H\in\mathcal A^n_{i-1}$ be the subspace provided by Lemma \ref{lem:Diamr_i}. Thus
    \[
    R_{n-i+1}(K) \leq R(P_{H^\bot}(K)) \leq \sqrt{\frac{n-i+1}{2(n-i+2)}} D(P_{H^\bot}(K)) \leq 2i\sqrt{\frac{n-i+1}{2(n-i+2)}} r_i(K)
    \]
    where the second inequality is Jung's Theorem applied to the $(n-i+1)$-dimensional convex body $P_{H^\bot}(K)$.
\end{proof}

We now compute the solution to \eqref{eq:rDR_inequality} by means of the diameter.
\begin{lemma}\label{lem:DiamIneqBGR}
  Let $K\in\mathcal K^3$ with $R(K)=1$ and $r(K) \leq 1/3$. Then
    \[
\begin{split}
& D(K)\geq \\
& \sqrt{\frac{1}{2}}\sqrt{-2r(K)^2+8r(K)+3 +\sqrt{3}\sqrt{4r(K)^4-16r(K)^3+26r(K)^2-16r(K)+3}}.
\end{split}
\]
\end{lemma}

\begin{proof}
    Doing some routine but tedious computations shows that the right-hand side of the inequality above is decreasing for $r(K)\in[0,1/3]$. Moreover, it is equal to $\sqrt{3}$ if $r(K)=0$. Hence, the inequality holds true when $D(K) \geq \sqrt{3}$. Now, if $D(K)<\sqrt{3}$ we may apply \eqref{eq:rDR_inequality} to $K$. Inserting $R(K)=1$ and solving for $D(K)$ yields the inequality.
\end{proof}

\begin{proof}[Proof of Theorem \ref{thm:53}]
    Let us apply Lemma \ref{lem:Diamr_i} to $K$ with $i=3$. Thus, there exists $H\in\mathcal A^5_2$ such that
    \[
    D(P_{H^\bot}(K)) \leq 6r_3(K).
    \]
    After a suitable rescaling and translation, we can assume that $R(P_{H^\bot}(K))=1$ with $P_{H^\bot}(K) \subseteq B_n\cap H^\bot$. Note that if $r_3(K) \geq 1/3$, we have that
    \[
    R_3(K) \leq R(P_{H^\bot}(K)) = 1 \leq 3 r_3(K)<3.5186 r_3(K).
    \]
    Hence, it remains to show the inequality in case of $r_3(K)<1/3$.

    By Proposition \ref{prop:opt}, let $p^1,\dots,p^k\in P_{H^\bot}(K)$, for some $k\in\{2,3,4\}$, and $S:=\mathrm{conv}(\{p^1,\dots,p^k\})$, be such that $R(S)=R(P_{H^\bot}(K))=1$. Since $S\subset P_{H^\bot}(K)$, by monotonicity we have that $r_{H^{\perp}}(S)\leq r_{H^{\perp}}(P_{H^{\perp}}(K))$ and $D(S)\leq D(P_{H^{\perp}}(K))$.

    Let $q^i\in K$ be such that $P_{H^{\perp}}(q^i)=p^i$, for every $i= 1,\dots,k$, and let $S':=\mathrm{conv}(\{q^1,\hdots,q^k\})$. Note that by \eqref{eq:inradiusProyectionidim} we get
    \[
    r_{H^\bot}(S) \leq r_{\mathrm{aff}(S')}(S') \leq r_{\mathrm{aff}(S')}(K\cap\mathrm{aff}(S')) \leq r_3(K)<\frac13.
    \]
Applying now Lemma \ref{lem:Diamr_i} to the convex body $K$ and Lemma \ref{lem:DiamIneqBGR} to the $3$-dimensional convex body $S$ we obtain
 \begin{align*}
       6r_3(K)\geq D(P_{H^{\perp}}(K))\geq D(S) \geq g(r_{H^{\perp}}(S))
    \end{align*}
    where
    \[
    g(x):=\sqrt{\frac{1}{2}}\sqrt{-2x^2+8x+3 +\sqrt{3}\sqrt{4x^4-16x^3+26x^2-16x+3}}.
    \]
    Again, as observed in the Proof of Lemma \ref{lem:DiamIneqBGR}, $g(x)$ is a decreasing function within $x\in[0,1/3]$, from which we can conclude that $6r_3(K) \geq g(r_3(K))$. Using one last time some algebraic tool, we solve this inequality in terms of $r_3(K)$ to derive
    \[
    r_3(K) \geq 0.28421 = 0.28421 R(S).
    \]
    Since $R_3(K) \leq R(P_{H^\bot}(K)) = R(S)$, we conclude $R_3(K) \leq 3.5186 r_3(K)$.    
\end{proof}

\begin{rmk}
    One may be tempted to think of \eqref{eq:rDR_inequality} in higher dimensions. The best-known bound in this regard is due to Dekster \cite{De85}. For every $K\in\mathcal K^n$ then
    \[
    r(K) \geq \sqrt{1-\frac{(n-1)}{2n}D(K)^2}.
    \]
    Following the same approach as in the proof of Theorem \ref{thm:53}, we would obtain for every $K\in\mathcal K^{2i-1}$ that
    \[
    R_i(K) \leq \sqrt{2i^2-2i+1}\, r_i(K).
    \]
    However, $\sqrt{2i^2-2i+1} < i+1$ only holds for $i<4$, and in those cases Perel'man's \eqref{eq:Pereln3i2} (dimension $3$) and Theorem \ref{thm:53} (dimension $5$) give better estimates. This shows that extending \eqref{eq:rDR_inequality} to higher dimensions is not only interesting in order to find complete systems of inequalities for the inradius, circumradius, and diameter, but also to improve \eqref{eq:PerelPuk}.
\end{rmk}

\section{Estimates for large values of $i$}\label{sec:largeValues}

We start this section with some technical lemmas. The first explains how to find supporting hyperplanes to some convex set containing a subspace that does not cut the interior of the same body.
\begin{lemma}\label{lem: hyperplane}
         Let $K\in\mathcal K^n$ and $H\in\mathcal{A}_i^n$ with $i\in\{1,\ldots,n-2\}$ such that $H\cap\mathrm{int}( K)=\emptyset$. Then, there exists $H'\in\mathcal{A}_{n-1}^n$ such that $H\subset H'$ and $H'\cap \mathrm{int}( K)=\emptyset.$
\end{lemma}
     \begin{proof}
         Let $p\in H$. Since $H\cap \mathrm{int}( K)=\emptyset$, we have 
         \[
         P_{(-p+H)^\perp}(H)\cap P_{(-p+H)^\perp}(\mathrm{int}( K))=\{p'\}\cap P_{(-p+H)^\perp}(\mathrm{int}( K))=\emptyset,
         \]
         where $p'=P_{(-p+H)^\perp}(p)$. Therefore, $p'$ and $P_{(-p+H)^\perp}\left(\mathrm{int}(K)\right)$ can be separated by $L$, a $(n-i-1)$-dimensional hyperplane in $(-p+H)^\perp$. Therefore, we get that $H\subset H'=\mathrm{aff}(\{H, L\})$ and $ H'\cap \mathrm{int}( K)=\emptyset.$
     \end{proof}
The following result is a generalization of an idea from \cite[Theorem 1.3]{GM17}. Here, we bound certain directional widths of a convex body by twice the same directional width of a certain intersection of $K$ with a hyperplane.
\begin{lemma}
    \label{lem:factortwo}
   Let $K\in\mathcal{K}^n$, $v,w\in \mathrm{bd}(K)$, and $a\in \R^n\setminus\{0\}$ be such that $w(K)=a^{\top}v-a^{\top}w$ with $v-w\in\langle a \rangle$.
   Furthermore, let $H=\frac{v+w}{2}+\langle a\rangle^{\perp}$. 
  Then, for $u\in \langle a \rangle ^{\perp}$, we have 
    \begin{equation*}
       h_K(u)+h_K(-u)\leq 2\left( h_{K\cap H} (u)+  h_{K\cap H} (-u) \right)
    \end{equation*}
\end{lemma}
\begin{proof} 
After a suitable rigid motion of $K$, we may assume that $a=e^1$, $u=e^2$ and $v=-w=v_1 e^1$. Then, $H=\langle e^1\rangle^{\perp}$ and $H$ passes through $\frac{v+w}{2}=0$.  

Let $x,z\in \mathrm{bd}(K)$ such that 
\begin{align*}
    h_{K\cap H} (e^2)&=(e^2)^{\top}x,\\
      h_{K} (e^2)&=(e^2)^{\top}z.
\end{align*}

Since $h_K(e^1)=v_1$ and $h_K(-e^1)=-v_1$, we have
\begin{equation}
\label{eq:l1}
      -v_1\leq z_1 \leq v_1 .
\end{equation}

By Lemma \ref{lem: hyperplane} applied to $K$ and $H'':= x+\langle \{e^3,\dots,e^n\}\rangle$, since $H''\cap\mathrm{int}(K)=\emptyset$, we obtain the existence of $H'\in\mathcal A^n_{n-1}$ 
with $H''\subset H'$ and $H'\cap\mathrm{int}(K)=\emptyset$. Note that the outer normal $a_x$ to $H'$ has to be orthogonal to $H''-x=\langle e^3,\dots,e^n\rangle$, and thus, we can assume that $a_x=\alpha e^1+e^2$ for some $\alpha\in\mathbb R$. In particular, $H'$ supports $K$ in $x$.
Then,  
\begin{align}
\label{eq:l0}
        h_K(a_x)=a_x^{\top}x=x_2=(e^2)^{\top}x.
\end{align}
It follows from \eqref{eq:l0} that
\begin{align}
       (e^2)^{\top}x&= h_K(a_x)\geq a_x^{\top}z=\alpha z_1+z_2,  \label{eq:l2}\\
         (e^2)^{\top}x&= h_K(a_x)\geq a_x^{\top}v=\alpha v_1,  \label{eq:l3}\\
          (e^2)^{\top}x&= h_K(a_x)\geq a_x^{\top}(-v)=-\alpha v_1.\label{eq:l4}
\end{align}
Using \eqref{eq:l2}, \eqref{eq:l1}, and then \eqref{eq:l3} or \eqref{eq:l4}, we obtain 
\begin{equation}
    \begin{split}
    h_K(e^2)=z_2&\leq (e^2)^{\top}x- \alpha z_1\\
    &\leq (e^2)^{\top}x +|\alpha| v_1\\
    &\leq (e^2)^{\top}x +(e^2)^{\top}x\\
    &=2 h_{K\cap H}(e^2). 
    \end{split}
\end{equation}
Analogously, $h_K(-e^2)\leq 2 h_{K\cap H}(-e^2)$, completing the proof. 
\end{proof}

The third lemma bounds from above the successive outer radius of $K$ by some multiple of the minimal width of a certain section of $K$.
\begin{lemma}
\label{lem:widthbound}
    Let $K\in\mathcal{K}^n$ and $1\leq i\leq n-1$. 
    There exists $H\in\mathcal{A}^n_{i}$ 
    such that
    \[
    R_{n-i+1}(K)\leq 2^{n-i-1}\sqrt{n-i+1}~w_{H}\left(K\cap H\right).
    \]
\end{lemma}
\begin{proof}
    Applying a rigid motion if needed, we can assume that $\pm(w(K)/2)e^1 \in K$ and $K$ is contained between the parallel supporting hyperplanes $\pm(w(K)/2)e^1 + \langle e^1\rangle^{\perp}$. 
    We denote by $w_0:=w(K)$ and $H_1:=\langle e^1\rangle^{\perp}$.

Now consider the $(n-1)$-dimensional body $K\cap H_1$. After another rigid motion if required, we may assume that its minimal width (in $H_1$) is attained in the direction $e^2$. Let $x^1,y^1\in\mathrm{relbd}(K\cap H_1)$ be such that
\begin{equation*}
    w_1:=w_{H_1}(K\cap H_1)=(e^2)^{\top} x^1 - (e^2)^{\top} y^1 
\end{equation*}
with $x^1-y^1\in\langle e^2\rangle$ (see \eqref{eq:widthParallel}).
Then, we define the $(n-2)$-dimensional affine subspace
\begin{equation*}
    H_2:=H_1\cap \left(\frac{x_1+y_1}{2}+\langle e^2 \rangle^{\perp}\right)=\frac{x_1+y_1}{2}+\langle e^1,e^2 \rangle^{\perp}.
\end{equation*}
We continue this construction iteratively. Hence, for $k\in[n-i]$, we consider the $(n-k)$-dimensional body $K\cap H_k$ and (after a suitable rotation) points $x^k,y^k\in\mathrm{relbd}(K\cap H_k)$ such that
\begin{equation*}
    w_k:=w_{H_k}(K\cap H_k)=(e^{k+1})^{\top} x^k - (e^{k+1})^{\top} y^k 
\end{equation*}
with $x^k-y^k\in\langle e^{k+1}\rangle$
and define
\begin{equation*}
    H_{k+1}:=\frac{x_k+y_k}{2}+\langle \{ e^1,\hdots,e^{k+1} \}\rangle^{\perp}.
\end{equation*}

Thus, we have
\begin{equation}
    h_{K\cap H_{k}}(e^{k+1})+ h_{K\cap H_{k}}(-e^{k+1})=w_k.
\end{equation}
We now apply Lemma \ref{lem:factortwo} for the $(n-k)$-dimensional body $K\cap H_{k}$, where $w_{H_k}(K\cap H_k)=(e^{k+1})^{\top} x^k - (e^{k+1})^{\top} y^k $, and where the subspace is precisely $H_{k+1}=\frac{x^k+y^k}{2}+\langle\{e^{k+2},\dots,e^n\}\rangle$, obtaining that
\begin{equation*}
    h_{K\cap H_{k}}(e^{j})+ h_{K\cap H_{k}}(-e^{j})\leq 2  ( h_{K\cap H_{k+1}}(e^{j})+ h_{K\cap H_{k+1}}(-e^{j}))
\end{equation*}
for $j\geq k+2$. In particular, for $j=k+2$, we obtain
\begin{equation}\label{eq:w2factor}
    \begin{split}
    w_k & \leq h_{K\cap H_{k}}(e^{k+2})+ h_{K\cap H_{k}}(-e^{k+2}) \\
        & \leq 2  ( h_{K\cap H_{k+1}}(e^{k+2})+ h_{K\cap H_{k+1}}(-e^{k+2})) = 2w_{k+1}.
    \end{split}
\end{equation}

By applying Lemma \ref{lem:factortwo} iteratively and due to the fact that $e^{k}\in \langle e^1,\dots,e^\ell\rangle^\bot$, for every $\ell=1,\dots,k-1$, we therefore obtain
\begin{equation}
\label{eq:boxsides}
    \begin{split}
          h_K(e^k)+h_K(-e^k)&\leq 2( h_{K\cap H_1}(e^k)+h_{K\cap H_1}(-e^k))\\
     &\leq 2^2( h_{K\cap H_2}(e^k)+h_{K\cap H_2}(-e^k)) \\
     &\leq \hdots\\
     &\leq 2^{k-1}( h_{K\cap H_{k-1}}(e^k)+h_{K\cap H_{k-1}}(-e^k))\\
     &= 2^{k-1}w_{k-1},
    \end{split}
\end{equation}
for $k\in[n-i+1]$.

The projection $P_{\langle e^1,\hdots, e^{n-i+1}\rangle}(K)$ is contained in the $(n-i+1)$-dimensional box 
\begin{equation*}
    \begin{split}
    &B:=\\
    &\left[-h_K(-e^1),h_K(e^1)\right]\times \left[-h_K(-e^2),h_K(e^2)\right]\times \hdots \times \left[-h_K(-e^{n-i+1}),h_K(e^{n-i+1})\right].
    \end{split}
\end{equation*}
By \eqref{eq:boxsides}, the side lengths of $B$ are upper bounded by  $w_0,2w_1,\hdots,2^{n-i}w_{n-i}$, and we know that the circumradius of a box is half the length of its diagonal. By \eqref{eq:w2factor}, we have $2^{k-1}w_{k-1}\leq 2^{n-i}w_{n-i}$ for every $k\in[n-i]$. Together we obtain, 
\begin{align*}
    R_{n-i+1}(K)&\leq R(P_{\langle e^1,\hdots, e^{n-i+1}\rangle}(K))\leq R(B)\\
&\leq\frac{1}{2}\sqrt{(w_o)^2+(2w_1)^2+\hdots+(2^{n-i}w_{n-i})^2}\\
&\leq\frac{1}{2}\sqrt{(n-i+1)(2^{n-i}w_{n-i})^2}\\
&=2^{n-i-1}\sqrt{n-i+1}~ w_{n-i}\\
&=2^{n-i-1}\sqrt{n-i+1}~ w_{K\cap H_{n-i}}(K\cap H_{n-i}). 
\end{align*}
Finally, choosing $H:=H_{n-i}$ concludes the lemma. 
\end{proof}

We can now finish the proof of Theorem \ref{thm:Main}.
\begin{proof}[Proof of right-hand-side of Theorem \ref{thm:Main}]
    Let $H\in \mathcal{A}^{n}_{i}$ 
    be as in Lemma \ref{lem:widthbound}.
    By Steinhagen's inequality \eqref{eq:Steinhagen} applied to 
    $K\cap H$, and since 
    $\max\{\sqrt{i}, (i+1)/\sqrt{i+2}\}\leq \sqrt{i+1}$, we conclude
    \begin{align*}
        R_{n-i+1}(K)&\leq 2^{n-i-1}\sqrt{n-i+1}~w_{H}\left(K\cap H\right)\\
        &\leq  2^{n-i}\sqrt{n-i+1}\sqrt{i+1}~r_{H}\left(K\cap H\right)\\
        &\leq  2^{n-i}\sqrt{n-i+1}\sqrt{i+1}~ r_i(K). 
    \end{align*}
\end{proof}

\begin{rmk}
    Let $m\in\mathbb N$ be fixed. Note that, on the one hand, for every $K\in\mathcal K^n$ and $i:=n-m$, the right-hand side of Theorem \ref{thm:Main} implies that
    \[
    \frac{R_{n-i+1}(K)}{r_i(K)} = \frac{R_{m+1}(K)}{r_{n-m}(K)} \leq 2^m\sqrt{(m+1)(n-m)}.
    \]
    On the other hand, \eqref{eq:simplex} implies that for a $n$-dimensional regular simplex $S_n$ we have that
    \[
    \frac{R_{m+1}(S_n)}{r_{n-m}(S_n)} = \sqrt{\frac{(m+1)(n-m)(n-m+1)}{n+1}}.
    \]
    This means that Theorem \ref{thm:Main} is asymptotically optimal when $i=n-m$ with $m$ fixed, up to the constant factor
    \[
    \lim_{n\rightarrow\infty} \frac{2^m\sqrt{(m+1)(n-m)}}{\sqrt{\frac{(m+1)(n-m)(n-m+1)}{n+1}}} = 2^m.
    \]
    We note that Perel'man's bound \eqref{eq:PerelPuk} $R_{m+1}(K)/r_{n-m}(K)\leq n-m+1$ gives a strictly worse asymptotic order of magnitude.
\end{rmk}
\clearpage
\textit{Acknowledgements:} The second author would like to thank the hospitality of the Technical University of Munich, where she completed a research stay together with Ren\'e Brandenberg and Mia Runge from April to June $2026$, and where most of this project was carried out.

\vspace{1cm}
\end{document}